\documentclass[english,11pt]{amsart}

\font \sevenrm=cmr7
\font \fiverm=cmr5
\usepackage[english]{babel}
\usepackage{amsmath}
\usepackage{amsfonts}
\usepackage{amssymb}
\usepackage{amsthm}
\usepackage{mathrsfs}
\usepackage[T1]{fontenc}
\usepackage{tikz}
\usepackage{shuffle}
\usepackage[all]{xy}
\usepackage{soul}

\usepackage{array,multirow,makecell}
\setcellgapes{1pt}
\makegapedcells

\newcolumntype{R}[1]{>{\raggedleft\arraybackslash }b{#1}}
\newcolumntype{L}[1]{>{\raggedright\arraybackslash }b{#1}}
\newcolumntype{C}[1]{>{\centering\arraybackslash }b{#1}}

\long\def\ignore#1{}
\newcommand\nc{\newcommand}
\def \restr#1{\mathstrut_{\textstyle |}\raise-6pt\hbox{$\scriptstyle #1$}}
\def \srestr#1{\mathstrut_{\scriptstyle |}\hbox to
-1.5pt{}\raise-4pt\hbox{$\scriptscriptstyle #1$}}
\nc{\wt}{\widetilde} \nc{\wh}{\widehat}
\nc\un{\mathbf{1}}
\nc{\mop}[1]{\mathop{\hbox {\rm #1} }\nolimits}
\nc{\gmop}[1]{\mathop{\hbox {\bf #1} }\nolimits}

\nc{\smop}[1]{\mathop{\hbox {\sevenrm #1} }\nolimits}
\nc{\ssmop}[1]{\mathop{\hbox {\fiverm #1} }\nolimits}
\nc{\mopl}[1]{\mathop{\hbox {\rm #1} }\limits}
\def\dbar{d\hskip-3pt \raise 4pt\hbox{-}}
\nc{\smopl}[1]{\mathop{\hbox {\sevenrm #1} }\limits}
\nc{\ssmopl}[1]{\mathop{\hbox {\fiverm #1} }\limits}

\newtheorem{theorem}{Theorem}
\newtheorem{definition}{Definition}
\newtheorem{lemma}{Lemma}

\newtheorem{remark}{Remark}

\def\nn{\mathbf{n}}
\def\aa{\mathbf{a}}
\def\bb{\mathbf{b}}
\def\bm{\mathbf{m}}
\def\ww{\mathbf{w}}
\def\blambda{\mathbf{\lambda}}
\def\var{\rm Var}
\def\M{\mathscr{M}}

\def\C{\mathbb{C}}
\def\Z{\mathbb{Z}}

\def\N{\mathbb{N}}

\def\di{\displaystyle}

\def\nn{\mathbf{n}}
\def\A{\mathbf{A}}

\usepackage{xcolor}

 \def\oo{\boldsymbol\omega}
\def\o{\omega}

\usepackage{amsaddr}

\author{Jacky Cresson}
\address{Université de Pau et des Pays de l'Adour - E2S, Laboratoire de Mathématiques et de leurs Applications, UMR CNRS 5142, Batiment IPRA, avenue de l'Université, 64000 Pau, France}

\author{Jordy Palafox}
\address{CY Tech, CY Cergy Paris Université, 2 Bd Lucien, 64000 Pau, France}

\email{jacky.cresson@univ-pau.fr, jordy.palafox@cyu.fr}

\begin{document}
\title{Variance of vector fields - Definition and properties}

\maketitle

\begin{abstract}

We give a self contained presentation of the notion of variance of a vector field introduced by Jean Ecalle and Bruno Vallet in \cite{ev} following a previous work of Jean Ecalle and Dana Schlomiuk in \cite{es}. We give complete proofs and definitions of various results stated in these articles. Following J. Ecalle and D. Schlomiuk, We illustrate the interest of the variance by giving a complete proof of the formulas for the mould defining the nilpotent part of a resonant vector field. 
\end{abstract}

\setcounter{tocdepth}{3}
\tableofcontents

\section{Introduction}

In \cite{es}, J. Ecalle and D. Schlomiuk study how classical objects attached to a given local vector field (or diffeomorphism) are affected by the action of an infinitesimal automorphism. Although this method plays a fundamental role to obtain informations about the nilpotent part of a vector field, the notion is not even named. This method is completely formalized by J. Ecalle and B. Vallet in \cite{ev} where they introduce the notion of {\it variance} of a vector field (or diffeomorphism). The main idea behind the notion of variance is that "the more direct the geometric meaning of a mould, the simpler its "variance" tends to be" (\cite{ev},p.270). Despite its importance, the definition proposed in (\cite{ev},p.269) is given  by formulas (see \cite{ev}, (3.28)-(3.30)) which are not fully proved in the paper. However, these formulas are fundamental in the proof of the analyticity of the correction and corrected form of a vector field (see \cite{ev}, Problem 2 p.258 and Remark 1 p.271) which is the main result of \cite{ev}. In particular, the formulas obtained for the variance of the correction are mainly used to prove that the Eliasson's phenomenon of compensation of small denominators can be avoided by a purely algebraic elimination of "illusory denominators" (see \cite{ev}, Remark 2 p.257). \\

In this paper, we give a self contained introduction to the notion of variance of a vector field as well as complete proofs of the main formulas. All the computations are done in the framework of the {\it mould formalism} as introduced by J. Ecalle in \cite{ec2} (see also \cite{cr2}). As an application of the variance, we provide complete proofs for the mould equations satisfied by the mould defining the Nilpotent part of a vector field as well as explicit computations. \\

The plan of the paper is organized as follows: In section \ref{infini}, we remind the notions of vector field in prepared form, alphabet associated to a given vector field and infinitesimal automorphisms. In section \ref{universal}, we remind the definition of mould and the property of universality as well as the notion of action of a universal mould. In \ref{secvariance}, we define the variance of a vector field and the associated notion of variance of a mould. We provide a complete proof of the formula for the variance of a mould. 
In section \ref{derivvar}, we define a family of derivations on moulds associated to the variance. We also prove the connection between the classical derivation $\nabla$ and this family of derivations. In \ref{varnilpo}, we shows how the variance of vector fields is applied to compute the mould associated to the nilpotent part of a resonant vector field. Section \ref{conclusion}, gives some perspectives of the present work.

\section{Vector fields and infinitesimal automorphisms}
\label{infini}

 We consider a vector field $X$ of $\C^d$ in {\it prepared} form, i.e. in a coordinates system such that
 \begin{equation}
     X=X_{lin} +\di\sum_{n\in A(X)} B_n ,
 \end{equation}
where the linear part is assumed to be diagonal given by 
\begin{equation}
    X_{lin} =\di\sum_{i=1}^d \lambda_i x_i \di\frac{\partial}{\partial x_i},
\end{equation}
with $\lambda_i \in \C$ for $i=1,\dots ,d$ and the {\it alphabet} $A(X)\subset \Z^d$, $n=(n_1 ,\dots ,n_d )$ is such that all $n_i \in \N$ unless one which can be equal to $-1$ and $B_n$ is a {\it homogeneous} operator of order $1$ and degree $\mid n \mid =n_1 +\dots +n_d$, i.e. satisfying for all monomials $x^m$, $m\in \N^d$ the equality
\begin{equation}
    B_n (x^m ) =\beta_{n,m} x^{n+m} ,
\end{equation}
where $\beta_{n,m} \in \C$. \\

We denote by $\blambda = (\lambda_ 1, \dots ,\lambda_d ) \in \C^d$ and for all $n\in Z^d$, we denote by $\langle \cdot ,\cdot \rangle$ the usual scalar product on $\C^d$. \\

A letter $n\in A(X)$ will be called {\it resonant} if and only if 
\begin{equation}
\langle n,\lambda \rangle =0.
\end{equation}

Let $\theta$ be a vector field. We denote by $\Theta =\exp (\theta )$ the exponential map defined for all vector fields $Y$ by 
\begin{equation}
    \exp (\theta ) (Y) =\di\sum_{r\geq 0} \di\frac{1}{r!} \di ad^r_{\theta} (Y ),
\end{equation}
where $ad_{\theta}$ is defined for all $Y$ by 
\begin{equation}
    ad_{\theta} (Y)=[Y,\theta ] .
\end{equation}
We have $\Theta^{-1} =\exp (-\theta )$.\\

If $Y$ is in the kernel of $ad_{\theta}$, i.e. $[Y,\theta ]=0$ then 
\begin{equation}
    \exp (\theta ) (Y)=Y .
\end{equation}
In particular, a vector field $Y$ is {\it resonant} with respect to $X_{lin}$ if 
\begin{equation}
[Y,X_{lin}] =0 .
\end{equation}

If $\theta$ is a resonant vector field, then we have $\exp (\theta ) (X_{lin} ) =X_{lin}$.\\

Let $B_c$ be an arbitrary homogeneous vector field of degree $c$. We denote by $U_{\epsilon,c}$ the {\it infinitesimal} automorphism of $\C [[ x ]]$ defined by 
\begin{equation}
    U_{\epsilon ,C} =\exp (\epsilon B_c ) ,
\end{equation}
where the exponential is defined by 
\begin{equation}
    \exp (\epsilon B_c )=\di\sum_{r\geq 0} \di\frac{1}{r!} \epsilon^r B_c^r ,
\end{equation}
with $B_c^r =B_c \circ \dots B_c$ the $r$-th composition of the differential operator $B_c$. \\

We denote by $X_{\epsilon}$ the vector field 
\begin{equation}
    X_{\epsilon} =U_{\epsilon ,c } X\, U_{\epsilon ,c}^{-1} .
\end{equation}

We denote by $\Psi_{\epsilon} : g \rightarrow g$ the map
\begin{equation}
    \Psi_{\epsilon} (X)=U_{\epsilon ,c } X\, U_{\epsilon ,c}^{-1} . 
\end{equation}

\section{Universal mould and action of a mould}
\label{universal}

The notion of {\it mould} is introduced by J. Ecalle in his seminal lecture notes "Les fonctions résurgentes" \cite{ec2}. In particular, moulds are used in the study of local vector fields or diffeomorphism in \cite{ec1} and subsequent works. A full introduction to this formalism can be found in \cite{cr2} where a foreword of J. Ecalle explains the general status and main applications of moulds. \\

A {\it mould} on a given alphabet $A$ can be formally defined as a function denoted by $M^{\bullet}$ from $A^*$ to $\C$, where $A^*$ denotes the set of words constructed on $A$, including the empty word denoted by $\emptyset$. We denote by $\mathscr{M} (A)$ the set of moulds on $A$. The set $\mathscr{M} (A)$ possesses a natural structure of linear space. The mould multiplication of two moulds $M^{\bullet}$, $N^{\bullet}$ in $\mathscr{M} (A)$ denoted by $P^{\bullet} =M^{\bullet} \times N^{\bullet}$ is defined for all $\nn \in A^*$ by 
\begin{equation}
    P^{\nn} =\di\sum_{\nn =\aa \bb} M^{\aa} N^{\bb} .
\end{equation}
The set of moulds $(\mathscr{M} (A) , +,\times )$ is then an associative algebra over $\C$.\\

In our application to vector fields $A \subset \Z^d$ and all the moulds we are considering can be seen as maps from $\Z^d$ to $\C$ by imposing $M^{\nn} =0$ for all $\nn \in \Z^d \setminus A$.\\

The definition of the variance of a vector field uses fundamentally that some moulds are 
{\it universal}. This notion is formally discussed in (\cite{cr2},p.374-375). A mould $M^{\bullet} \in \mathscr{M} (A)$, $A\subset \Z^d$, is said to be {\it universal} if there exists a family of complex functions $F_r :\C^r \rightarrow \C$ and a map $\omega : \Z^d \rightarrow \C$ such that for all $\nn \in A^*$, $\nn =n_1 \dots n_r$ we have 
\begin{equation}
    \oo ( \nn )= (\omega (n_1 ),\dots ,\omega (n_r ) ),
\end{equation}
and
\begin{equation}
    M^{\nn} = F_r (\oo (\nn ) ) .
\end{equation}

As an example, the {\it mould of linearization} for vector fields in prepared form (see \cite{cr2},p.375) denoted by $Na^{\bullet}$ is associated to the following family $\mathbf{La}=(La_r)_{r\geq 1}$, $r\in \N$ of complex valued functions defined for all $r\geq 1$ by $La_r : \C^r \rightarrow \C$, 
\begin{equation}
    La_r (x_1 ,\dots ,x_r ) =\di\frac{1}{(x_1 +\dots +x_r ) (x_1 +\dots +x_{r-1} ) \dots x_1} ,
\end{equation}
for all $x=(x_1 ,\dots ,x_r ) \in \C^r \setminus Sa_r$ where 
\begin{equation}
    Sa_r =\{ x_1 =0\} \cup \{ x_1 +x_2 =0\} \cup \dots \{ x_1 +\dots +x_r =0\} .
\end{equation}
For an alphabet $A$ which is non-resonant, i.e. such that for all $\nn \in A^*$, we have $\oo (\nn )=\langle \nn ,\mathbf{\lambda} \rangle \neq 0$, we have for a word $\nn \in A^*$ of length $r$, 
\begin{equation}
    Na^{\nn} =La_r (\oo (\nn ) ) . 
\end{equation}

A universal mould retains its shape although the alphabet is different. Indeed, let $M^{\bullet} \in \mathscr{M} (A)$ be a mould satisfying the universality property. Then, the mould $M^{\bullet}$ keeps a meaning for an arbitrary alphabet $\tilde{A} \subset \Z^d$. For any words $\ww \in \tilde{A}^*$ of length $r$, we have 
\begin{equation}
    M^{\ww} =F_r (\oo (\ww ) ) ,
\end{equation}
which is well defined.\\

This property can be used by introducing the notion of {\it action} induced by a mould.\\

Let $M^{\bullet}$ be a mould satisfying the universality property. For any alphabet $A \subset \Z^d$, the generating function of $M^{\bullet}$ defined on $A$ is given by
\begin{equation}
    \Psi_{M^{\bullet}} (A):=\di\sum_{\nn \in A^*} M^{\nn } \nn .
\end{equation}

As the generating function of a mould $M^{\bullet}$ satisfying the universality property is defined for an arbitrary alphabet $A\subset \Z^d$, we can define an {\it action} of $M^{\bullet}$ on an arbitrary vector field $X$ as follows (see \cite{ev}, p.260).\\

\begin{definition}[Action of a mould on vector fields]
Let $X$ be a vector field which generates the alphabet $A(X)$ and the family of differential operators $\left \{ B_n \right \}_{n\in A(X)}$. The action of $M^{\bullet}$ on $X$ denoted by $Act^{M^{\bullet}} (X)$ is the differential operator 
\begin{equation}
    Act^{M^{\bullet}} (X):=\di\sum_{\nn \in A(X)^*} M^{\nn } B_{\nn} ,
\end{equation}
where for a given word $\nn \in A^*$, $\nn =n_1 \dots n_r$ we have 
\begin{equation}
    B_{\nn} =B_{n_r} \dots B_{n_1} ,
\end{equation}
where the product must be understood as the composition of differential operators.
\end{definition}

Depending on the {\it symmetry} (see \cite{cr2}, p. 346-350) satisfied by the mould, the nature of the image of $X$ under the action of $M^{\bullet}$ is different.\\

In the following, we consider moulds $M^{\bullet}$ which are {\it alternal} (see \cite{cr2}, 4.1. p.331), i.e. such that $M^{\emptyset} =0$ and for all $\aa, \bb \in A^*\setminus \{\emptyset \}$, we have 
\begin{equation}
    \di\sum_{\nn \in sh (\aa ,\bb ) } M^{\nn} =0 ,
\end{equation}
where $sh (\aa ,\bb )$ called the {\it shuffling} of $\aa$ and $\bb$ is the "set of sequences $\nn$ that can be obtained by intermingling the sequences $\aa$ and $\bb$ under preservation of their internal order" (see \cite{ev},p.261).\\

If the mould $M^{\bullet}$ satisfies the universality property and is alternal, then $Act^{M^{\bullet}}$ transforms a vector field in a vector field.

\section{Variance of a universal mould}
\label{secvariance}

We can define the {\it variance} of a universal alternal mould $M^{\bullet}$ as follows:

\begin{definition}[Variance of a vector field]
Let $A_{\epsilon} (X)$ denotes the alphabet generated by $X_{\epsilon}$. The action of $M^{\bullet}$ on $X_{\epsilon}$ can be written as 
\begin{equation}
   X_{M,A_{\epsilon}} :=Act^{M^{\bullet}} (X_{\epsilon} ) =X_{M,A} +\epsilon \var_C (X_{M,A}) +o(\epsilon ^2 ) ,
\end{equation}
where the vector field $\var_c (X_{M,A} )$ defined is called the variance of $X_{M,A}$ under $U_{\epsilon ,B_c}$. 
\end{definition}

We prove the following result which is stated in (\cite{ev}, (3.30) p.269). 

\begin{theorem}[Variance of a universal mould]
    Let $X$ be in prepared form and $A(X)$ the associated alphabet. We denote by $B_c$ a homogeneous vector field of degree at least $2$. Let $A_c =A(X)\cup \{c\}$. We have 
    \begin{equation}
    \var_c (X_{M ,A} )=\di\sum_{\nn \in A_c^*} Var_c (M^{\bullet} )^{\nn}  B_{\nn} ,
\end{equation}
with 
\begin{equation}
Var_c (M^{\bullet} )^{\nn} =\left ( \di\sum_{i=1}^{l(\nn )} \var_{c,i} (M)^{\nn} \right )
\end{equation}
where for $i=1,\dots ,l(\nn )$ we have 
\begin{equation}
    Var_{c,i} (M^{\bullet} )^{\nn} = \left \{ 
    \begin{array}{ll}
    0 & \ \ \ \mbox{\rm if}\ n_i \not = c ,\\
    \omega (c) M^{\nn} + M^{conf_i (\nn )} -M^{conb_i (\nn )} &\ \ \ {\rm if}\ \ n_i =c,
    \end{array}
    \right .
\end{equation}
where the two operators $conf_i$ and $conb_i$ are respectively the {\it forward (resp. backward) contraction of position $i$} defined by
\begin{equation}
    conf_i (\nn ) =\nn^{<i} (n_i +n_{i+1}) \nn^{>i+1}\ \ 
    \mbox{and}
    \ \ 
    conb_i (\nn ) =\nn^{<i-1} (n_{i-1} +n_i) \nn^{>i} ,
\end{equation}
with
\begin{equation}
    \nn^{<i} =n_1 \dots n_{i-1} \ \ \mbox{and}\ \ \ 
    \nn^{>i} =n_{i+1} \dots n_r .
\end{equation}
Moreover, if $c\not\in A(X)$ and $\nn$ contains at least two times the letter $c$ then \begin{equation}
     Var_{c,i} (M^{\bullet} )^{\nn} =0 .
\end{equation}

\end{theorem}

It must be noted that, due to the universality of the mould $M^{\bullet}$, the variance of the mould $M^{\bullet}$ is by definition not dependant on the alphabet, contrary to the variance of $X_{M,A}$.

\begin{proof}
In order to compute $Var_c (X_{M,A})$ we first obtain the letter of the new alphabet generated by $X_{\epsilon}$. By définition, we have
\begin{equation}
\left .
\begin{array}{lll}
    X_{\epsilon} & = & X +\epsilon [B_c ,X] +o(\epsilon^2 ) ,\\
     & = & X_{lin} +\di\sum_{n\in A} B_n + \epsilon \left ( [B_c ,X_{lin}] +
\di\sum_{n\in A} [B_c ,B_n ] \right ) +o(\epsilon^2 ) .
\end{array}
\right .
\end{equation}
We have that 
\begin{equation}
    [B_c ,X_{lin}]=\omega (c) B_c ,
\end{equation}
and the Lie bracket $[B_c ,B_n ]$ is a homogeneous vector field of order $n+c$. \\

As a consequence, we have four kind of homogeneous operators $B_{\epsilon ,m}$, $m\in A_{\epsilon}$, depending on the properties of $c$: \\

First, we have of course 
\begin{equation}
    B_{\epsilon ,c}=\epsilon \omega (c) B_c .
\end{equation}
We denote by $D(c)$ the set defined by
\begin{equation}
    D(c)=\{ n\in A, \ c-n \in A \} ,
\end{equation}
If $n\in D (c)$, we denote by $B_{\epsilon ,n}$ the vector field 
\begin{equation}
    B_{\epsilon , n} =B_n +\epsilon [B_c ,B_{n-c} ] +o(\epsilon^2 ) .
\end{equation}
If $n\not\in D(c)$, we have  
\begin{equation}
    B_{\epsilon , n} =B_n +o(\epsilon^2 ), 
\end{equation}
and operators of order $m=n+c$ given by 
\begin{equation}
    B_{\epsilon ,m} =\epsilon [B_c ,B_n ] +o(\epsilon^2 ) .
\end{equation}
The other homogeneous vector fields $B_{\epsilon ,m}$ with $m\in A_{\epsilon}$ are of order at least $\epsilon^2$. \\

We have by definition
\begin{equation}
    X_{M,A_{\epsilon}}=\di\sum_{\bm \in A_{\epsilon}} M^{\bm} B_{\epsilon ,\bm} .
\end{equation}

As $Var_c (X_{_{\epsilon}})$ is the part of order one in $\epsilon$ of $X_{M,\epsilon}$,  we can focus on the composition of the four previous operators. We then have   
\begin{equation}
\left .
\begin{array}{lll}
    X_{M,A_{\epsilon}} & = & X_{M,A} \\
    & & +\epsilon 
    \di\sum_{n\in D(c) , \bm \in A^*} \di\sum_{\aa \bb =\bm} M^{\aa n \bb} B_{\aa} [B_c , B_{n-c} ] B_{\bb} \\
    & & +\epsilon \di\sum_{n\not\in D(c) , \bm \in A^*} \di\sum_{\aa \bb =\bm} M^{\aa (n+c)\bb} B_{\aa} [B_c ,B_n ] B_{\bb}  \\
    & & +\epsilon \di\sum_{\bm \in A^*} \di\sum_{\aa \bb =\bm} M^{\aa c\bb} B_{\aa} (\omega (c) B_c ) B_{\bb} \\
    & & +o(\epsilon^2 ) .    
\end{array}
\right .
\end{equation}
Developing the Lie bracket $[B_n ,B_c]$, with $n\in A$, as $B_n B_c -B_c B_n$, we see that we have an expression with operators of the form 
\begin{equation}
    B_{\aa} B_c B_{\bb}
\end{equation}
with $\aa, \bb \in A^*$. As a consequence, we have 
\begin{equation}
    Var_c (X_{M,a} ) =\di\sum_{\nn\in A^*} \di\sum_{i=0}^r  Var_c (M)^{\nn^{\leq i} c\nn^{>i}} B_{\nn^{\leq i} c\nn^{>i}} ,
\end{equation}
which can be rewritten as a sum over $A_c^*$ as follows 
\begin{equation}
    Var_c (X_{M,a} ) =\di\sum_{\nn\in A_c^*} Var_c (M)^{\nn} B_{\nn} ,
\end{equation}

If $\nn =\nn^{< i} c \nn^{>i}$, meaning that the $i$-th letter of $\nn$ is $c$, then we have two sources of the operator $B_{\nn^{< i} c\nn^{>i}}$. The operator $B_{\nn^{< i} c\nn^{>i}}$ can be obtained as 
\begin{equation}
   \omega (c) M^{\nn^{< i} c\nn^{>i}} B_{\nn^{< i}} B_c B_{\nn^{>i}}
\end{equation}
or from a term of the form
\begin{equation}
M^{\aa (n+c) \bb} B_{\aa} [B_c ,B_n ] B_{\bb} ,
\end{equation}
for $n\in A$. Developing the operator $B_{\aa} [B_c ,B_n] B_{\bb}$ as   
\begin{equation}
M^{\aa (n+c) \bb} B_{\aa} B_c B_n B_{\bb} - M^{\aa (n+c) \bb} B_{\aa} B_n B_c  B_{\bb},
\end{equation}
Taking $\aa =\nn^{<i}$ and $n \bb =\nn^{>i}$ we see that we have a term
\begin{equation}
  M^{\nn^{<i} (c+n_{i+1}) \nn^{>i+1}} B_{\nn <i} B_c B_{\nn^{>i}}
\end{equation}
and taking $\aa n =\nn^{<i}$ and $\bb=\nn^{>i}$, we have a term 
\begin{equation}
-  M^{\nn^{<i-1} (n_{i-1} +c) \nn^{>i}} B_{\nn <i} B_c B_{\nn^{>i}}
.
\end{equation}
Let $\nn$ be a word containing only one time the letter $c$. Then the coefficient in front of the operator $B_{\nn}$ is given by 
\begin{equation}
    \omega (c) M^{\nn} +M^{\nn^{<i} (n_i+n_{i+1}) \nn^{>i+1}} -M^{\nn^{<i-1} (n_{i-1} +n_i) n_{>i}}  
     .
\end{equation}
The quantity $Var_{c,i} (M^{\bullet} )$ encodes the previous computations in a unified way.\\

To finish the proof we have two cases:\\

\begin{itemize}
\item If $c\not\in A$ and $\nn$ contains at least two letters $c$, then $Var_c (M)^{\nn} =0$ as well as when $\nn$ does not contains the letter $c$. 

\item If $c\in A$, then for a given word of the form $\nn^1 c \nn^2 c\dots \nn^{k-1} c \nn^k$ we have to sum all the contributions coming from the bracket $B_{\epsilon ,c}$ at the different places where $c$ appears in the sequence, i.e. that $Var_c (M)$ will be given by
\begin{equation}
    Var_c (M)^{\nn} =\di\sum_{i=1}^r Var_{c,i} (M)^{\nn} .
\end{equation}
\end{itemize}
This concludes the proof.
\end{proof}

\section{Derivation associated to the variance}
\label{derivvar}

Following the work of J. Ecalle and D; Schlomiuk (\cite{es}, p.1421, (3.9)), We define the operator $Var_c :\M (A_c^* ) \rightarrow \M (A_c^* )$ by
\begin{equation}
    Var_c (M^{\bullet} )^{\nn} :=\di\sum_{i=1}^{l(\nn )} Var_{c,i} (M^{\bullet} )^{\nn} .
\end{equation}

Let $A$ be an alphabet. An operator $D$ on $\M (A )$ is a {\it derivation} if for all $M^{\bullet},N^{\bullet}$ in $\M (A)$, we have 
\begin{equation}
    D (M^{\bullet} \times N^{\bullet} ) =D(M^{\bullet} ) \times N^{\bullet} +M^{\bullet} \times D(N^{\bullet} ) .
\end{equation}

The following theorem was stated without proof in \cite{es}:

\begin{theorem}
    The operator $Var_c$ is a derivation on $\M (A_c^* )$.
\end{theorem}

\begin{proof}
    We have to compute for two arbitrary moulds $M^{\bullet}$ and $N^{\bullet}$ of $\M (A_c^* )$ the quantity $Var_c (M^{\bullet} \times N^{\bullet} )$. We have 
    \begin{equation}
        Var_c (M^{\bullet} \times N^{\bullet})^{\nn} :=\di\sum_{i=1}^{l(\nn )} Var_{c,i} (M^{\bullet} \times N^{\bullet})^{\nn} .
    \end{equation}
Let $\nn$ be such that $n_i =c$ and $n_j\not =c$ for $j\not= i$ then
    \begin{equation}
        Var_{c,j} (M^{\bullet} \times N^{\bullet})^{\nn} =0 
    \end{equation}
    for $j\not=i$ and
    \begin{equation}
        Var_{c,i} (M^{\bullet} \times N^{\bullet})^{\nn} =\omega (c) (M^{\bullet} \times N^{\bullet} )^{\nn} + (M^{\bullet} \times N^{\bullet} )^{conf_i (\nn )} -(M^{\bullet} \times N^{\bullet} )^{conb_i (\nn )}
    \end{equation}
    As a consequence, we obtain
    \begin{equation}
    \label{var0}
        Var_c (M^{\bullet} \times N^{\bullet})^{\nn} =\omega (c) (M^{\bullet} \times N^{\bullet} )^{\nn} + (M^{\bullet} \times N^{\bullet} )^{conf_i (\nn )} -(M^{\bullet} \times N^{\bullet} )^{conb_i (\nn )} .
    \end{equation}
Denoting $r=l (\nn )$ the lenght of $\nn$, We have    
\begin{equation}
    (M^{\bullet} \times N^{\bullet} )^{conf_i (\nn )} =\di\sum_{\aa \bb = conf_i (\nn ) } M^{\aa} N^{\aa} =\sum_{j=1}^{i-1} M^{\nn^{<j}} N^{conf_{i-j+1} (\nn^{\geq j} )} +
    \sum_{j=i}^{r} M^{conf_i (\nn^{<j})} N^{\nn^{\geq j}}
\end{equation}
and moreover, we obtain 
\begin{equation}
\left .
\begin{array}{lll}
\left (     Var_c (M^{\bullet})\times N^{\bullet} \right ) ^{\nn} & = & 
\di\sum_{j=1}^r \di Var_c (M^{\bullet})^{\nn <j} N^{\nn^{\geq j}} ,\\
 & = & \di\sum_{j=1}^{i-1} Var_c (M^{\bullet})^{\nn <j} N^{\nn^{\geq j}} 
 + \di\sum_{j=i}^{r} Var_c (M^{\bullet})^{\nn <j} N^{\nn^{\geq j}} .\\
\end{array}
\right .
\end{equation}
As $\nn^{<j}$ does not contain the letter $c$ for $j\leq i-1$, we have $Var_C (M^{\bullet})^{\nn^{<j}} =0$, and we obtain
\begin{equation}
\left (     Var_c (M^{\bullet})\times N^{\bullet} \right ) ^{\nn} = \di\sum_{j=i}^{r} Var_c (M^{\bullet})^{\nn^{<j}} N^{\nn^{\geq j}} .
\end{equation}
As we have only one letter $c$ at the $i$-th place in $n^{<j}$ for $j\geq i$, we obtain
\begin{equation}
    Var_{c,k} (M^{\bullet} )^{\nn^{<j}} =0 \ \mbox{\rm for}\  k=1,\dots , j-1,\ \ j\not= i ,  
\end{equation}
and
\begin{equation}
Var_{c,i} (M^{\bullet} )^{\nn^{<j}} = \omega (c) M^{\nn^{<j}} + M^{conf_i (\nn^{<j} )} -M^{conb_i (\nn^{<j} )} .
\end{equation}
As a consequence, we obtain for $j\geq i$ the equality
\begin{equation}
    Var_c (M^{\bullet})^{\nn^{<j}} =\omega (c) M^{\nn^{<j}} + M^{conf_i (\nn^{<j} )} -M^{conb_i (\nn^{<j} )} . 
\end{equation}
Replacing $Var_c (M^{\bullet})^{\nn^{<j}}$ by its expression, we obtain
\begin{equation}
\label{var1}
\left .
\begin{array}{lll}
\left (     Var_c (M^{\bullet})\times N^{\bullet} \right ) ^{\nn} & = & \di\sum_{j=i}^{r}  \left (  \omega (c) M^{\nn^{<j}} + M^{conf_i (\nn^{<j} )} -M^{conb_i (\nn^{<j} )} \right )  N^{\nn^{\geq j}} ,\\
 & =  & \omega (c) \di\sum_{j=i}^r M^{\nn^{<j}} N^{\nn^{\geq j}}
 + \di\sum_{j=i}^{r}  \left ( M^{conf_i (\nn^{<j} )} -M^{conb_i (\nn^{<j} )} \right )  N^{\nn^{\geq j}} .
\end{array}
\right .
\end{equation}
In the same way, we can compute $\left ( M^{\bullet} \times Var_c (N^{\bullet} ) \right ) ^{\nn}$. Indeed, we have 
\begin{equation}
\left (     M^{\bullet}\times Var_c (N^{\bullet} ) \right ) ^{\nn} = \di\sum_{j=i}^{r} M^{\bullet})^{\nn^{<j}} Var_c (N^{\bullet} ){\nn^{\geq j}} .
\end{equation}
As we have only one letter $c$ at the $i-j+1$-th place in $n^{\geq j}$ for $j\leq i$, we obtain
\begin{equation}
    Var_{c,k} (N^{\bullet} )^{\nn^{\geq j}} =0 \ \mbox{\rm for}\  k=1,\dots , r-j+1,\ \ j\not= i-j+1 ,  
\end{equation}
and
\begin{equation}
Var_{c,i-j+1} (M^{\bullet} )^{\nn^{\geq j}} = \omega (c) N^{\nn^{<j}} + N^{conf_{i-j+1} (\nn^{<j} )} -N^{conb_{i-j+1} (\nn^{<j} )} .
\end{equation}
As a consequence, we obtain for $j\leq i$ the equality
\begin{equation}
    Var_c (N^{\bullet} ) ^{\nn^{\geq j }} = \omega (c) N^{\nn^{\geq j}} + N^{conf_{i-j+1} (\nn^{\geq j} )} -N^{conb_{i-j+1} (\nn^{\geq j} )} . 
\end{equation}
Replacing $Var_c (N^{\bullet})^{\nn^{\geq j}}$ by its expression, we obtain
\begin{equation}
\label{var2}
\left .
\begin{array}{lll}
\left (     M^{\bullet}\times Var_c (N^{\bullet} ) \right ) ^{\nn} & = &  \di\sum_{j=1}^{i} M^{\nn^{<j}} \left ( \omega (c) N^{\nn^{\geq j}} + N^{conf_{i-j+1} (\nn^{\geq j} )} -N^{conb_{i-j+1} (\nn^{\geq j} )} \right ) ,\\
 & = & \omega (c) \di\sum_{j=1}^{i} M^{\nn^{<j}} N^{\nn^{\geq j}}\\
 &    & 
 +\di\sum_{j=1}^{i} M^{\nn^{<j}} M^{\nn^{<j}} \left ( N^{conf_{i-j+1} (\nn^{\geq j} )} -N^{conb_{i-j+1} (\nn^{\geq j} )} \right ) .
\end{array}
\right .
\end{equation}
Regrouping \eqref{var1} and \eqref{var2}, we have 
\begin{equation}
    \left (  Var_c (M^{\bullet})\times N^{\bullet} + 
    M^{\bullet}\times Var_c (N^{\bullet} ) \right ) ^{\nn} =
    \omega (c) \left ( M^{\bullet} N^{\bullet} \right )^{\nn} + (M^{\bullet} \times N^{\bullet} )^{conf_i (\nn )} - 
    (M^{\bullet} \times N^{\bullet} )^{conb_i (\nn )}
\end{equation}
which gives using \eqref{var0}
\begin{equation}
\left (  Var_c (M^{\bullet})\times N^{\bullet} + 
    M^{\bullet}\times Var_c (N^{\bullet} ) \right ) ^{\nn} =   Var_c (M^{\bullet} \times N^{\bullet})^{\nn} 
\end{equation}
This concludes the proof.
\end{proof}

In (\cite{es},p.1421), J. Ecalle and D. Schlomiuk state an equality between two operators. Namely, they introduce the operator denoted $\nabla$ and defined by
\begin{equation}
    \nabla (M^{\bullet} )^{\nn} = \omega (\nn ) M^{\bullet} . 
\end{equation}
This operator plays a central role in the linearization problem of vector fields (see \cite{ec1,cr2} for more details). \\

The following result is states without proof in (\cite{es},p.1421):

\begin{theorem}
\label{nablavart}
We have the equality
\begin{equation}
\label{nablavar}
    \nabla :=\di\sum_{a \in A} Var_{a} .
\end{equation}
\end{theorem}

The operator $\nabla$ is known to be a derivation on $\M (A)$ (see for example \cite{cr2}) directly by computations. An alternative proof follows from the formula \eqref{nablavar} and the fact that $Var_a$ is a derivation for all $a\in A$. 

\begin{proof}
Let $\nn \in A^*$ be given, $\nn =n_1 \dots n_r$. For all $a\not= n_i$, $i=1\dots ,r$, we have $Var_a (M^{\bullet} )^{\nn} =0$. Taking into account that a letter can appear many times in the word $\nn$, We then have 
\begin{equation}
    \di\sum_{a\in A} Var_a (M^{\bullet} )^{\nn} =\di\sum_{j=1}^r Var_{j, n_j} (M^{\bullet} )^{\nn} .
\end{equation}
We than have 
\begin{equation}
\left .
\begin{array}{lll}
    \di\sum_{a\in A} Var_a (M^{\bullet} )^{\nn} & = &  
    \di\sum_{j=1}^r \omega (n_j ) M^{\nn} + M^{\nn^{<r-1} (n_{r-1} +n_r)}\\ 
& &     +\di\sum_{j=1}^{r-1} M^{\nn^{<j} (n_j +n_{j+1} ) \nn^{>j+1}} -
    \di\sum_{j=2}^{r} M^{\nn^{<j-1} (n_{j-1} +n_j ) \nn^{<j}}\\
& &    - M^{\nn^{<r-1} (n_{r-1}+n_r)}.
\end{array}
\right .
\end{equation}
The extremal terms are compensating as well as the two sums on the second line. We then obtain 
\begin{equation}
\di\sum_{a\in A} Var_a (M^{\bullet} )^{\nn}     =  \omega (\nn ) M^{\nn} .
\end{equation}
We then conclude that for all $\nn \in A^*$, we have 
\begin{equation}
\di\sum_{a\in A} Var_a (M^{\bullet} )^{\nn} =\nabla (M^{\bullet} )^{\nn} .    
\end{equation}
This concludes the proof.
\end{proof}

\section{Variance and the Nilpotent part of resonant vector fields}
\label{varnilpo}

\subsection{The Nilpotent part of a resonant vector field}

Let $X$ be in prepared form. There exists a decomposition of $X$ as 
\begin{equation}
\label{decomp}
    X=X_{dia} +X_{nil},
\end{equation}
where $X_{dia}$ and $X_{nil}$ are formal vector fields and 
\begin{equation}
    [X_{dia},X_{nil}]=0 ,
\end{equation}
into a {\it diagonalizable} part $X_{dia}$ and a nilpotent part $X_{nil}$. The diagonalizable part $X_{dia}$ is formally linearizable and $X_{nil}$ has no linear component. The decomposition is {\it chart invariant}, i.e. that for any substitution operator $\Theta$ we have 
\begin{equation}
    (\Theta X\Theta^{-1} )_{dia} = \Theta X_{dia} \Theta^{-1} \ \ \mbox{\rm and}\ \ (\Theta X\Theta^{-1} )_{nil} = \Theta X_{nil} \Theta^{-1} .
\end{equation}

In \cite{es}, J. Ecalle and D. Schlomiuk prove that the nilpotent and diagonalizable part have a mould expension of the form 
\begin{equation}
    X_{dia} = \di\sum_{\nn \in A^*} Dia^{\nn} B_{\nn}\ \ \mbox{and}\ \ 
    X_{nil} = \di\sum_{\nn \in A^*} Nil^{\nn} B_{\nn} ,
\end{equation}
where the mould $Dia^{\bullet}$ and $Nil^{\bullet}$ have to be computed. \\

The decomposition \eqref{decomp} implies that
\begin{equation}
    I^{\bullet} =Dia^{\bullet} +Nil^{\bullet} .
\end{equation}

In (\cite{es}, p.1422), they state without proof the following result for which some arguments are given in (\cite{es},p.1424):

\begin{theorem}\label{thm_var_nil}
    The mould $Nil^{\bullet}$ satisfies the functional equation
    \begin{equation}
    \label{varcnil}
        Var_c (Nil^{\bullet} ) =I_c^{\bullet} \times Nil^{\bullet} -Nil^{\bullet} \times I_c^{\bullet} 
    \end{equation}
    with $Nil^{\emptyset} =0$ and $Nil^{n} =1$ if $n\in Res (A)$.\\

    We have also
    \begin{equation}
        \nabla (Nil^{\bullet} ) = I^{\bullet} \times Nil^{\bullet} -Nil^{\bullet}\times I^{\bullet} .
    \end{equation}
\end{theorem}

\begin{proof}
    Using the same notations as in the previous section, we have 
    \begin{equation}
        (X_{nil} )_{\epsilon} =X_{nil} +\epsilon [B_c ,X_{nil}] +o(\epsilon ) ,
    \end{equation}
    and
    \begin{equation}
        (X_{\epsilon} )_{nil} = X_{nil} +\epsilon Var_c (X_{nil} ) +o(\epsilon ) .
    \end{equation}
As the nilpotent (as well as the diagonalizable) part is chart invariant, we have 
\begin{equation}
    (X_{nil} )_{\epsilon} = (X_{\epsilon} )_{nil} .
\end{equation}
As a consequence, we obtain
\begin{equation}
    Var_c (X_{nil} ) = [B_c ,X_{nil}] .
\end{equation}
The Lie bracket of $B_c$ and $X_{nil}$ can be written using the mould formalism over $A_c^*$. Indeed, we have 
\begin{equation}
    [B_c , X_{nil} ] =B_c X_{nil} -X_{nil} B_c .
\end{equation}
As $X_{nil} =\di\sum_{\nn \in A^*} Nil^{\nn} B_{nn}$, we obtain 
\begin{equation}
    [B_c ,X_{nil} ]= \di\sum_{\nn \in A^*} Nil^{\nn} B_c B_{\nn} - \di\sum_{\nn \in A^*} Nil^{\nn} B_{\nn} B_c .
\end{equation}
Let us denote by $M^{\bullet}$ the mould defined for all $\ww \in A_c^*$ by 
\begin{equation}
    M^{\ww} =\left \{ 
    \begin{array}{lll}
    Nil^{\nn} & \ \ & \mbox{if}\ \ww = c\nn,\ \nn \in A^* ,\\
    -Nil^{\nn} & & \mbox{if}\ \ww = \nn c,\ \nn \in A^* ,\\
    0 & & \mbox{otherwise}
    \end{array}
    \right .
\end{equation}
Let $I_c^{\bullet}$ be the mould defined on $A_c^*$ by 
\begin{equation}
    I_c^{\ww} =\left \{ 
    \begin{array}{lll}
    1 & \ \ & \mbox{if}\ \ww = c,\\
    0 & & \mbox{otherwise} .
    \end{array}
    \right .
\end{equation}
Then we have the following equality
\begin{equation}
    M^{\bullet} =I_c \times Nil^{\bullet} -Nil^{\bullet} \times I_c^{\bullet} .
\end{equation}
We deduce that 
\begin{equation}
    Var_c (X_{nil} ) =\di\sum_{\ww \in A_c^*} \left ( I_c \times Nil^{\bullet} -Nil^{\bullet} \times I_c^{\bullet} \right )^{\ww} B_{\ww} .
\end{equation}
As 
\begin{equation}
    Var_c (X_{nil} ) =\di\sum_{\ww \in A_c^*} Var_c (Nil^{\bullet} )^{\ww} B_{\ww} ,
\end{equation}
we finally obtain
\begin{equation}
    Var_c (Nil^{\bullet} )^{\bullet} = I_c \times Nil^{\bullet} -Nil^{\bullet} \times I_c^{\bullet} .
\end{equation}
This concludes the proof of the first formula. The second one using the derivation $\nabla$ follows from theorem \ref{nablavart}. \\

For the initial conditions on the mould $Nil^{\bullet}$, as $X_{nil}$ commutes with $X_{dia}$ it contains only resonant terms and as $X=X_{dia} +X_{nil}$ we have to impose 
\begin{equation}
    Nil^{\emptyset} =0\ \ \mbox{and}\ \ Nil^{n} =1\ \mbox{if}\ n\in Res (A)\ \mbox{and}\ 0\ \mbox{otherwise.} 
\end{equation}
\end{proof}

We can directly check that the mould $Dia^{\bullet}$ has to satisfy the equation
    \begin{equation}
        Var_c (Dia^{\bullet} ) =I_c^{\bullet} \times Dia^{\bullet} -Dia^{\bullet} \times I_c^{\bullet} 
    \end{equation}
but with different initial conditions. Indeed, we have 
    \begin{equation}
        Var_c (Dia^{\bullet} ) =Var_c (I^{\bullet} ) -Var_c (Nil^{\bullet} ) 
    \end{equation}
by linearity of the derivation $Var_c$. Then we obtain
\begin{equation}
\left .
\begin{array}{lll}
        Var_c (Dia^{\bullet} ) & = & Var_c (I^{\bullet} )- I_c^{\bullet} \times Nil^{\bullet} +Nil^{\bullet} \times I_c^{\bullet} ,\\
         & = & Var_c (I^{\bullet} ) - I_c^{\bullet} \times I^{\bullet} + I_c^{\bullet} \times Dia^{\bullet} +I^{\bullet} \times I_c^{\bullet} - Dia^{\bullet} \times  I_c^{\bullet} .
\end{array}
\right .
\end{equation}
We have 
\begin{equation}
    Var_c (I^{\bullet} )^{\nn} =\left \{ 
    \begin{array}{l}
    \omega (c)  \ \mbox{if}\ \nn =c ,\\
    1 \ \mbox{if}\ \nn =cm ,\\
    -1 \ \mbox{if}\ \nn = mc ,\\
    0\ \mbox{otherwise.}
    \end{array}
    \right .
\end{equation}
and 
\begin{equation}
    (I_c^{\bullet} \times I^{\bullet} )^{\nn}=
    \left .
    \begin{array}{l}
    1\ \mbox{if}\ \nn =cm ,\\
    0 \ \mbox{otherwise.}
    \end{array}
    \right .
\ \     
    (I^{\bullet} \times I_c^{\bullet})^{\nn} =
    \left .
    \begin{array}{l}
    1\ \mbox{if}\ \nn =mc ,\\
    0 \ \mbox{otherwise.}
    \end{array}
    \right .
\end{equation}
If $c$ is such that $\omega (c)=0$ then $Var_c (I^{\bullet} )=0$ then we obtain for all $\nn =A_c$ that 
\begin{equation}
    Var_c (I^{\bullet} ) - I_c^{\bullet} \times I^{\bullet} +I^{\bullet} \times I_c^{\bullet}
    = 0 .
\end{equation}
As a consequence, for all $c\in Res (A)$, we obtain 
\begin{equation}
    Var_c (Dia^{\bullet} )=I_c \times Dia^{\bullet} -Dia^{\bullet}\times I_c^{\bullet} . 
\end{equation}

\subsection{Explicit computation of the mould $Nil^{\bullet}$}

The mould $Nil^{\bullet}$ is not easy to compute even if we have the mould equation \eqref{varcnil}. Let us compute it for sequences of length $\leq 4$ for which J. Ecalle and D. Schlomiuk provide a table (see \cite{es},p.1481-1482) but without the details of the computations. However, these computations are interesting by itself as they use the variance rules as a key ingredient. In this Section, we give explicit proof for all the formula.\\

As remarked by J. Ecalle and B. Vallet in (\cite{ev},p.271), the variance provides an "{\it overdetermined} induction" for the computation of a mould as it can be applied for different letters of the same word. \\

Indeed, for a word $\nn =n_1 \dots n_r$, we can take $c=n_1$ in the variance formula, so that if $n_i\not= n_1$ for $i=2,\dots ,r$, we obtain
\begin{equation}
\label{cal1}
    \omega (n_1) Nil^{n_1 \dots n_r} +Nil^{(n_1 +n_2) n_3 \dots n_r} = Nil^{n_2 \dots n_r} .
\end{equation}

The same computation can be performed with $c=n_r$ and $n_i\not= n_r$ for $i=1,\dots ,r-1$, and we obtain 
\begin{equation}
\label{calr}
\omega (n_r) Nil^{n_1 \dots n_r} -Nil^{n_1 \dots n_{r-2} (n_{r-1} +n_r)} = -Nil^{n_1 \dots n_{r-1}} .
\end{equation}
For $i\not= \{ 1,r \}$, $c=n_i$ and $n_j\not =n_i$, $j=1,\dots ,r$, $j\not= i$, we obtain 
\begin{equation}
\label{cali}
\omega (n_i) Nil^{n_1 \dots n_r} -Nil^{n_1 \dots (n_{i-1} +n_i ) \dots n_r} +Nil^{n_1 \dots (n_i +n_{i+1}) \dots n_r} = 0 .
\end{equation}

We have for all $\nn \in A^*$ that 
\begin{equation}
    \omega (\nn) Nil^{\nn} = Nil^{\nn^{>1}} - Nil^{\nn^{<r}} .
\end{equation}

In the following, for a given word $\nn =n_1 \dots n_r$, we denote by $\oo$ the vector of weights $(\omega_1 ,\dots ,\omega_r )$ where $\omega_i =\omega (n_i )$ and by $\mid \oo \mid$ the quantity $\mid \oo\mid  = \omega_1 +\dots +\omega_r$.

\begin{remark}
    In \cite{es} and \cite{ev} and other articles like \cite{ec1,ec2}, J. Ecalle writes moulds not on the alphabet $A$ generated by the vector field but by $\Omega$ which is the set of weight generated by the letter $n\in A$, i.e. 
    \begin{equation}
        \Omega =\{ \omega (n),\ n \in A \} .
    \end{equation}. 
We then denote by $F_{\omega}$, $\omega \in \Omega$, the differential operator
\begin{equation}
    F_{\omega} =\di\sum_{n\in A,\ \omega (n)=\omega} B_n .
\end{equation}
    However, there is no one-to-one correspondance between a letter and a weight. For example, for a two dimensional vector field with a linear part given by $\mathbf{\lambda} =(i,-i )$, $i^2 =-1$, the weight $0$ can be realized by any homogeneous differential operator of order $(m , m)$ where $m\in \N$, $m\geq 1$. As a consequence, working with moulds on $\Omega^*$ induces confusion on the computations and formulas as long as one wants to deal with composition of homogeneous operators. \\

    However, most of the formula proved on $A^*$ for the linearization or prenormalisation of vector fields persist on $\Omega^*$. This is due to the fact that the formula (see \cite{cr2}, Corollaire V.74 p. 369)
    \begin{equation}
        X_{lin} B_{\nn}  =\mid \omega (\nn ) \mid B_{\nn} + B_{\nn} X_{lin} 
    \end{equation}
    is preserved on $\Omega^*$, i.e. denoting by $B_{\oo}$ the differential operator 
    \begin{equation}
        F_{\oo} =F_{\omega_r} \dots F_{\omega _1} , 
    \end{equation}
    for $\oo =\omega_1 \dots \omega _r$, we have 
    \begin{equation}
        X_{lin} F_{\oo} =\mid \oo \mid F_{\oo} + F_{\oo} X_{lin} ,
    \end{equation}
    using the linearity of $X_{lin}$. \\
    
    As a consequence, the mould equations for the conjugacy of vector fields retain their form from $A^*$ to $\Omega^*$ as can be seen from the proof of the (\cite{cr2} Théorème V.72 p.368) given in (\cite{cr2}, p.370).
\end{remark}

It must be noted that for a full resonant word, i.e. a word $\nn =n_1 \dots n_r$ such that $\oo =(0,\dots ,0)$, the previous set of equations does not allow to compute the value of the mould $Nil^{\nn}$. As a consequence, we have to fix the value as initial conditions. \\

This can be done by assuming that the mould $Nil^{\bullet}$ has to be alternal (see \eqref{alternal}) in order that $X_{nil}$ is a vector field. In that case, we must have
\begin{equation}
Nil^{\emptyset} =0 ,
\end{equation}
and
\begin{equation}
    \di\sum_{\nn \in sh (0, 0^r )} Nil^{\nn} =0
\end{equation}
for all $r\geq 1$, $0^r = \underset{r \text{ times}}{\underbrace{0\dots 0}}$. As for any $\nn \in sh (0,0^r )$, we have $\nn=0^{r+1}$, we deduce that
\begin{equation}
    (r+1) M^{0^{r+1}} =0 ,
\end{equation}
which implies 
\begin{equation}
    M^{0^{r+1}} =0 .
\end{equation}

These relations provide a flexible way to compute the mould $Nil^{\bullet}$. We give explicit formula for words of length $1$ to $3$.

\begin{lemma}
Let $n\in A$ then 
\begin{center}
\begin{tabular}{|R{2cm} | L{1.5cm} |}
 \hline   $\nn$ & $Nil^{\bullet}$ \\
 \hline   $\omega (n) \not= 0$ & 0 \\
\hline   $\omega (n)=0$ & 1\\
\hline
\end{tabular}
\end{center}
\end{lemma}

\begin{proof}
If $n\in A$ is such that $\omega (n)\not =0$, we obtain as $Nil^{\emptyset} =0$ that 
\begin{equation}
    \omega (n) Nil^{n} =0 ,
\end{equation}
and as a consequence, $Nil^n =0$. 
\end{proof}

\begin{lemma}
Let $\nn$ be a word of length $2$. For non-resonant words, i.e. $\mid \oo \mid \not= 0$, we have
\begin{center}
\begin{tabular}{|R{8cm} | L{1.5cm} |}
 \hline   $\nn$ & $Nil^{\bullet}$ \\
 \hline   $\omega_1 \not=0$, $\omega_2 \not= 0$ & $0$ \\
\hline   $\oo =(0,\omega )$, $\omega \not= 0$ & $-\omega^{-1}$\\
\hline    $\oo =(\omega ,0 )$, $\omega \not= 0$ & $\omega^{-1}$\\
\hline
\end{tabular}
\end{center}

For resonant words, i.e. $\mid \oo \mid =0$, we have 

\begin{center}
\begin{tabular}{|R{8cm} | L{1.5cm} |}
 \hline   $\nn$ & $Nil^{\bullet}$ \\
\hline $\oo =(0,0)$ & $0$\\
\hline $\oo =(\omega ,-\omega )$, $\omega \not= 0$ & $-\omega^{-1}$\\
\hline
\end{tabular}
\end{center}
\end{lemma}

\begin{proof}
Let $\nn = n_1 n_2$ be a non-resonant word with $\mid \oo \mid =\omega \not= 0$. Then if $n_1\in Res (A)$, i.e. $\omega_1 =0$ then $n_2 \in A$ satisfies $\omega_2 =\omega$ and we obtain using  \eqref{calr} we have 
\begin{equation}
\omega  Nil^{\nn} -Nil^{n_1+n_2} = -Nil^{n_1} ,
\end{equation}
which gives
\begin{equation}
    Nil^{\nn} =\omega^{-1} \left ( -Nil^{n_1} +Nil^{n_1+n_2} \right ) .
\end{equation}
As $\omega_1 =0$ and $\omega (n_1 +n_2 )=\omega$, we deduce that $Nil^{n_1} =1$ and $Nil^{n_1 +n_2} =0$. As a consequence, we obtain
\begin{equation}
    Nil^{\nn} =-\omega^{-1} .
\end{equation}
In the same way, when $\oo =(\omega ,0)$, with $\omega \not= 0$ using \eqref{cal1}, we obtain
\begin{equation}
    \omega  Nil^{\nn} +Nil^{n_1 +n_2} = Nil^{n_2} .
\end{equation}
As $Nil^{n_2} =1$ and $\omega (n_1 +n_2)=\omega \not= 0$, we have $Nil^{n_1+n_2} =0$ and 
\begin{equation}
    Nil^{\nn} = \omega^{-1} .
\end{equation}

For resonant words, we have two cases. If $\oo = (0,0 )$ then $Nil^{\nn} =0$ by assumption. Otherwise, we have $\oo =(\omega , -\omega )$ with $\omega \not= 0$. Using \eqref{cal1} we obtain
\begin{equation}
    \omega  Nil^{\nn} +Nil^{n_1 +n_2} = Nil^{n_2} .
\end{equation}
As $\omega (n_1 +n_2 ) =0$ and $\omega (n_2)\not=0$, we have $Nil^{n_1+n_2} =1$ and $Nil^{n_2} =0$ so that 
\begin{equation}
    Nil^{\nn} =-\omega^{-1} .
\end{equation}
\end{proof}

\begin{lemma}
Let $\nn$ be a word of length $3$. For non resonant word, i.e. $\mid \oo \mid \not= 0$, we have: 

\begin{center}
\begin{tabular}{|C{8cm} | L{2.5cm} |}
 \hline  $\nn$ & $Nil^{\bullet}$ \\
 \hline  $\oo =(\omega_1 ,\omega_2 ,\omega_3 )$, $\omega_1 ,\omega_2 ,\omega_3 \not=0$, $\omega_1 +\omega_2 \not= 0$, $\omega_1 +\omega_3 \not=0$, $\omega_2 +\omega_3 \not= 0$ & $0$ \\
\hline  $\oo = (\omega , \tilde{\omega} , -\tilde{\omega} )$, $\omega\not= 0$, $\tilde{\omega}\not= 0$, $ \o + \tilde{\o} \neq 0$ & $-\omega^{-1} \tilde{\omega}^{-1}$ \\
\hline  $\oo = (\tilde{\omega} , \omega ,-\tilde{\omega} )$, $\omega\not= 0$, $\tilde{\omega}\not= 0$, $ \o + \tilde{\o} \neq 0$ & $0$ \\
\hline  $\oo = (\omega , -\omega,  \tilde{\omega}  )$, $\omega\not= 0$, $\tilde{\omega}\not= 0$, $ \o + \tilde{\o} \neq 0$ & $\omega^{-1} \tilde{\omega}^{-1}$ \\
\hline   $\oo =(0,\omega ,\tilde{\omega} )$, $\omega+\tilde{\omega} \not= 0$, $\omega\not= 0$, $\tilde{\omega}\not= 0$ & $\omega^{-1}(\omega+\tilde{\o})^{-1}$\\
\hline    $\oo =(\omega ,0,\tilde{\omega} ) $, $\omega+\tilde{\omega} \not= 0$, $\omega\not= 0$, $\tilde{\omega}\not= 0$ & $-\omega^{-1}\tilde{\o}^{-1}$\\
\hline $\oo =(\omega ,\tilde{\omega} ,0) $, $\omega+\tilde{\omega} \not= 0$, $\omega\not= 0$, $\tilde{\omega}\not= 0$ & $\tilde{\omega}^{-1}(\o+\tilde{\o})^{-1}$\\
\hline  $\oo =(\omega ,0,0)$, $\omega\not=0$ & $\omega^{-2}$\\
\hline  $\oo =(0,\omega ,0)$, $\omega\not=0$ & $2\omega^{-2}$\\
\hline  $\oo =(0,0 ,\omega )$, $\omega\not=0$ & $-\omega^{-2}$\\
\hline
\end{tabular}
\end{center}

For resonant words $\nn$, i.e. $\mid \oo \mid =0$, we have:

\begin{center}
\begin{tabular}{|R{9cm} | L{1.5cm} |}
 \hline   $\nn$ & $Nil^{\bullet}$ \\
 \hline   $\oo = (0 ,0 ,0)$ & $0$ \\
 \hline   $\oo = (0 ,\omega ,-\omega )$ & $-\omega^{-2}$ \\
\hline   $\oo = (\omega  ,0 ,-\omega )$ & $2\omega^{-2}$ \\
\hline   $\oo = (\omega  ,-\omega  ,0)$ & $-\omega^{-2}$\\
\hline
\end{tabular}
\end{center}
\end{lemma}

\begin{proof}
We decompose the study by looking for non-resonant words and resonant words.
Firstly, we consider \textbf{non-resonant words}. Let $\nn=n_1n_2n_3$ be such that $\o(\nn)\neq 0$ then we use the formula : 
$$\o(\nn)Nil^\nn = Nil^{n_2n_3} - Nil^{n_1n_2}.$$
Dealing with this decomposing in two words $n_2n_3$ and $n_1n_2$, notice that $n_1$ and $n_3$ do not interact with each other then we consider :\\
\textbf{Case 1} : If $\omega (n_2 n_3)=0$ then $\omega (n_1 )\not =0$ and two sub-cases have to be considered.
        \begin{enumerate}
            \item If $\omega (n_2 ) =0$ then $\omega (n_3 )=0$ and $Nil^{n_2 n_3} =0$ and $Nil^{n_1 n_2 } = -\omega_1^{-1}$. We obtain for $\nn$ with a weight $\oo = (\omega_1  , 0 ,0)$ and $\omega_1 \not= 0$ that $Nil^{n_1 n_2 n_3} =\omega_1^{-2}$.

            \item If $\omega (n_2)\not= 0$ then $\omega (n_3 ) =-\omega (n_2)$ and $\o(n_1) \neq 0$ and we have two subcases :
            \begin{enumerate}
                \item If $\omega (n_1 n_2) \not= 0$ then $Nil^{n_2 n_3} = -\omega_2^{-1}$ and $Nil^{n_1 n_2} =0$. We obtain for $\nn$ with a weight $\oo =(\omega_1  ,  \omega_2 , -\omega_2 )$ and $\mid \oo \mid \not= 0$ that $Nil^{\nn} =-\omega_1^{-1} \omega_2^{-1}$.
                \item If $\omega (n_1 n_2) = 0$, then $Nil^{n_2n_3}= -\o_2^{-1}$ and $Nil^{n_1n_2}= \o_2^{-1}$ and we obtain : $\o(\nn)Nil^\nn = -2(\o_2)^{-2} $.
            \end{enumerate}
        \end{enumerate}

 \textbf{Case 2} : If $\omega (n_2 n_3 )\not =0$ then two sub-cases have to be considered. 
    \begin{enumerate}
        \item If $\omega (n_1)=0$ we have three cases. 
        \begin{enumerate}
            \item If $\omega (n_2 )=0$ and $\o(n_3) \neq 0$ then $\omega (n_3)\not= 0$ and we have $Nil^{n_2 n_3} = -\omega_3^{-1}$ and $Nil^{n_1 n_2} = 0$ so that for a word $\nn$ with a weight $\oo = (0 ,0 ,\omega_3 )$ with $\omega_3\not= 0$, we have $Nil^{\nn} = -\omega_3^{-2}$. 

            \item If $\omega (n_2)\not= 0$ and $\o(n_3)= 0$, then we have $Nil^{n_1n_2}=-\o_2^{-1}$ and $Nil^{n_2n_3}=\o_2^{-1}$. We obtain $Nil^\nn=-2\o_2^{-2}$ because $\o(\nn)=\o_2.$

            \item If $\omega (n_2)\not= 0$ and $\omega (n_3)\not= 0$, then we have $Nil^{n_1n_2}=-\o_2^{-1}$ and $Nil^{n_2n_3}=0$. We obtain $Nil^\nn=-\o_2^{-1}(\o_2+\o_3)^{-1}$. 
        \end{enumerate}
        \item If $\omega (n_1)\neq 0$, then we have the four sub-cases : 
        \begin{enumerate}
            \item If $\o_1+\o_2= 0$  then we must have $\o_3\neq 0$ ($\o_3=0$ is a resonant case) and $ Nil^\nn = -\o_1^{-1}\o_3^{-1}.$

            \item  If $\o_1+\o_2\neq 0$ such that $\o_2 \neq 0$  and $\o_3=0$ then we obtain $Nil^{n_1n_2}=Nil^{n_2n_3}=0$ and $Nil^\nn=0.$

            \item If $\o_1+\o_2\neq 0$ such that $\o_2 = 0$  and $\o_3\neq 0$ then we obtain $Nil^{n_2n_3}=-\o_3^{-1}$ and $Nil^{n_1n_2}=\o_1^{-1}$. We obtain $Nil^\nn=-2\o_1^{-1}\o_3^{-1}.$
            
            \item If $\o_1+\o_2\neq 0$  and $\o_1+\o_3\neq 0$ then we have $Nil^{n_1n_2}=0$ and $Nil^{n_2n_3}=0$ and $Nil^\nn=0.$
        \end{enumerate} 
        
    \end{enumerate}

Now, we consider $\nn$ \textbf{resonant}, that is $\mid \omega (\nn )\mid =0$ we can organize the computations with respect to the number of resonant letters using the formula of the variance applied to the mould $Nil^\bullet$ (see Theorem 4). \\
If we have two resonant letters, due to $\mid \oo \mid =0$ the three letter are resonant and $Nil^{\nn} =0$ by assumption.\\
If we have one resonant letter, the two remaining ones have opposite weights. We have three cases:
\begin{itemize}
    \item If $\oo =(0,\omega ,-\omega )$ then $Nil^{(n_1 +n_2) n_3} =-\omega^{-1}$ and $Nil^{n_1 (n_2+n_3)} = 0$ so that $Nil^{\nn} = -\omega^{-2}$ using $Var_2$. 

    \item If $\oo =(\omega ,0,-\omega )$ then $Nil^{(n_1 +n_2 ) n_3 } = -\omega^{-1}$ and $Nil^{n_2 n_3} = \omega^{-1}$. We obtain $Nil^{\nn} = 2\omega^{-2}$ using $Var_1$.

    \item If $\oo =(\omega , -\omega ,0 )$ then $Nil^{(n_1 +n_2 ) n_3} = 0$ and $ Nil^{n_1 (n_2 +n_3 )} = -\omega^{-1}$. We obtain $Nil^{\nn} = -\omega^{-2}$ using $Var_1$. 
\end{itemize}
This completes the proof.

\end{proof}

\section{Conclusion and perspectives}
\label{conclusion}

In this article, we have given the definition and first properties of the variance of a vector field following as closely as possible the previous work of J. Ecalle and D. Schlomiuk \cite{es} where this notion was introduced but not named as variance and J. Ecalle and B. Vallet \cite{ev} where the notion is formalized and fully used for the analysis of a new object called the {\it correction} of a vector field. \\

However, several results stated in these two articles have to be discussed and supported by complete proofs. In particular, the two following problems have to be analyzed:\\

\begin{itemize}
\item In (\cite{es},p. 1432) the analytic properties of the mould $Nil^{\bullet}$ are studied. In particular, it is proved that the form of the mould is preserved under arborification. We refer to \cite{cmp} for an introduction to arborification of functional equations on moulds and a complete proof of this result. \\

\item The phenomenon of "non-appearance of multiple small denominators" exhibit in \cite{ev} in order to prove the analyticity  of the correction and the analytical linearizability of the corrected form is based on a careful study of the {\it variance rules} obtained for the mould of the correction (see \cite{ev},p.290) and their behavior under arborification. The study of the mould of the correction as well as its properties under arborification will be the subject of a forthcoming article.
\end{itemize}

\end{document}